\documentclass{article}
\usepackage{amsmath,amssymb,amsthm}

\allowdisplaybreaks

\newcommand{\R}{\mathbb{R}}

\newtheorem{Proposizione}{Proposition}
\newtheorem{Teorema}{Theorem}

\newtheorem{Corollario}{Corollary}
\title{On-flow and strong solutions\\
to Killing-type equations}

\author{Gianluca Gorni\\
Universit\`a di Udine\\
Dipartimento di Matematica e Informatica\\
via delle Scienze~208, 33100 Udine, Italy\\
\tt{gianluca.gorni@uniud.it}
\and
Gaetano  Zampieri\\
Universit\`a di Verona\\
Dipartimento di Informatica\\
strada Le Grazie 15, 37134 Verona, Italy\\
\tt{gaetano.zampieri@univr.it}}

\date{June 17, 2014}

\begin{document}
\maketitle

\begin{abstract}
If we impose infinitesimal invariance up to a boundary term of the action functional for Lagrangian ordinary differential equations, we are led to Killing-type equations, which are related to first integrals through Noether theorem. We review the ``on-flow'' and ``strong'' interpretations of the Killing-type equation, and for each we detail the complete explicit structure of the solution set in terms of the associated first integral. We give examples that reappraise the usefulness of the ``on-flow'' solutions. Finally, we describe an equivalent alternative approach to variational invariance.
\end{abstract}



This work was done under the auspices of the INDAM (Istituto Nazionale di Alta Matematica). The authors are grateful to Prof.~Giuseppe Gaeta for helpful discussion.

\section{Introduction}\label{introduction}

Suppose we are given a smooth Lagrangian function $L(t,q,\dot q)$, with $t\in\R$, $q,\dot q\in\R^n$. The variational principle for Lagrangian dynamics posits that
\begin{equation}\label{variationalPrinciple}
  \delta\int_{t_1}^{t_2}L\bigl(t,q(t),\dot q(t)\bigr)dt=0,
\end{equation}
which is equivalent to the \emph{Euler-Lagrange equation}
\begin{equation}\label{Lagrange}
  \frac{d}{dt}\partial_{\dot q}L\bigl(t,q(t),\dot q(t)\bigr)
  -\partial_{q}L\bigl(t,q(t),\dot q(t)\bigr)=0.
\end{equation}
We will assume that the Euler-Lagrange equation can be put into normal form
\begin{equation}\label{normalformLagrange}
  \ddot q=\Lambda(t,q,\dot q),
\end{equation}

Following closely the notation of Sarlet and Cantrijn~\cite{SarletCantrijn} (except that $\dot q$ is an independent variable from the outset), we consider an \emph{infinitesimal transformation} in the $(t,q)$ space given by
\begin{equation}\label{infinitesimalTransformation}
  \bar t=t+\varepsilon \tau(t,q,\dot q),\qquad
  \bar q=q+\varepsilon\xi(t,q,\dot q).
\end{equation}
This transformation is said to leave the action integral (infinitesimally) \emph{invariant up to boundary terms} (using the nomenclature recommended by P.G.L.~Leach), if a function $f(t,q,\dot q)$ exists, such that for the given smooth curve $t\mapsto q(t)$ we have
\begin{multline}\label{infinitesimalInvariance}
  \int_{\bar t_1}^{\bar t_2}L\Bigl(\bar t,\bar q(\bar t),
  \frac{d\bar  q}{d\bar t}(\bar t)\Bigr)d\bar t=
  \int_{t_1}^{t_2}L\bigl(t,q(t),\dot q(t)\bigr)dt+{}\\
  +\varepsilon\int_{t_1}^{t_2}
  \frac{df}{dt}\bigl(t,q(t),\dot q(t)\bigr)dt
  +O(\varepsilon^2),
\end{multline}
which is equivalent to the following \emph{Killing-type equation} for ODEs:
\begin{equation}\label{Killing-type}
  \tau\partial_t L
  +\partial_q L\cdot \xi
  +\partial_{\dot q}L\cdot
  \bigl(\dot \xi-\dot q\dot\tau\bigr)
  +L\dot\tau=
  \dot f.
\end{equation}
This is the same formula as Sarlet and Cantrijn~\cite{SarletCantrijn}, formula~(9) p.~471. We will write $\partial_{t}, \partial_{q}, \partial_{\dot q}$ for the partial derivative and gradients, $\dot x$ for the total time derivative of the function~$x$, and $x\cdot y$ will denote the ordinary scalar product of $x,y\in\R^n$.

Noether's theorem states that equation~\eqref{Killing-type} is a sufficient condition so that the function
\begin{equation}\label{firstintegral}
  N=f-L\tau-\partial_{\dot q}L\cdot\bigl(\xi-\dot q\tau\bigr)
\end{equation}
(\cite[p.~471, formula~(11)]{SarletCantrijn}) be a \emph{constant of motion} for the solutions to the Lagrange equation~\eqref{Lagrange}.

The function~$L(t,q,\dot q)$ will be given, and we solve Killing-type equation~\eqref{Killing-type} for the triple $(\tau,\xi,f)$. The equation in the terse form~\eqref{Killing-type} is open to at least three interpretations that we know of, differing on what the independent variables are and on how to treat the $\ddot q$ terms.

The most restrictive approach is when we seek $\tau,\xi, f$ as functions of $(t,q)$ only: ($\tau(t,q),\xi(t,q),f(t,q)$). Since in this case $\ddot q$ does not appear, the independent variables are $t,q,\dot q$ only, and we want the equation to hold identically. We will be concerned with this approach mainly in Section~\ref{strongIndependentOfDotQ}.

If we instead allow full dependence of $\tau,\xi,f$ on~$\dot q$, the expanded-out form of the total time derivatives $\dot\xi,\dot \tau,\dot f$ will contain~$\ddot q$. We distinguish two alternatives:

\begin{itemize}

\item \emph{Strong form}: we treat $\ddot q$ as just another independent variable and require the equation to hold for all $t,q,\dot q,\ddot q$. This way the Hamiltonian action functional will be infinitesimally invariant along all smooth trajectories~$q(t)$. The strong form was introduced by Djukic~\cite{Djukic}, and studied also by Kobussen~\cite{Kobussen}. Since the equation depends linearly on~$\ddot q$, it is equivalent to a system of $n+1$ equations that do not contain~$\ddot q$, but we will not exploit this fact in the sequel.

\item \emph{On-flow form}: we replace every occurrence of $\ddot q$ with $\Lambda(t,q,\dot q)$ from the normalized Lagrange equation~\eqref{normalformLagrange}, and require the resulting equation to hold for all $t,q,\dot q$. This way the Hamiltonian action functional will be infinitesimally invariant as in~\eqref{infinitesimalInvariance} only along the Lagrangian motions~$q(t)$, a condition which is however enough for the expression~\eqref{firstintegral} to be a first integral.

\end{itemize}

In Sections~\ref{KillingSection} and~\ref{strongIndependentOfDotQ} of this work we give a full description of the structure of the solution sets of the Killing-like equation, in the above senses, assuming that we are given both $L$ and the first integral~$N$. This problem is also known as ``reverse Noether theorem''. The main results of Sections~\ref{KillingSection} and~\ref{strongIndependentOfDotQ}, specially for the strong form, are basically contained already in Sarlet and Cantrijn's 1981 paper~\cite{SarletCantrijn}, where they are deduced as a corollaries of the theory of~$d\theta$-symmetries. Here we make a direct derivation, which we hope will be helpful to some readers. 

In Section~\ref{examplesSection} we illustrate the general theory with examples. In Subsection~\ref{FreeParticle} we re-examine the free particle's notorious ``non-noetherian'' symmetries that were ``noetherized'' by P.G.L.~Leach~\cite{Leach2} by substituting a new Lagrangian for the usual one. We will see what we can say about those symmetries from the point of view of on-flow and strong solutions without switching Lagrangian.

In Subsection~\ref{superintegrableSubSection} we search for classes of superintegrable systems among the ones of the form $\ddot x=-G(x)$, $\ddot y=-G'(x)y$, by making a suitable ansatz on the solution triple $(\tau,\xi,f)$ in the on-flow form. For the sake of clarity we will stop short of pursuing the method beyond known territory. The on-flow form of Killing-like equation was dismissed by most authors, the reason being that it has too many solution. Our point is that a large solution set can work to our advantage when we do not know exactly the form of the system and of the first integral, because an ansatz has more chances of catching a solution when the solutions are plenty.

Subsection~\ref{LRLSubSection} is devoted to the Laplace-Runge-Lenz vector conservation of the classical Kepler problem. We review some known formulas for solutions to Killing-like equation in either on-flow or strong interpretation, and propose some new ones in dimension~3, that feel simpler to us.

In Section~\ref{differentTimeChangeSection} we return to an alternative way of performing time change in the action integral, that we proposed in a recent paper~\cite{GorniZampieri}. We show how this approach leads to a different, but equivalent, Killing-like equation and to some formulas that match closely with formulas written in a 1972 paper~\cite{Candotti} by Candotti, Palmieri and Vitale.

\section{Structure of the solution sets}\label{KillingSection}

Equation~\eqref{Killing-type} is referred to as ``Killing-type'' because of the important particular case when the Lagrangian function is a quadratic form in the~$\dot q$ variable: $L=\frac12\dot q\cdot A(q)\dot q$, with $A(q)$ symmetric $n\times n$ non-singular matrix. The Lagrange equation~\eqref{normalformLagrange} reduces to the equation of geodesics. Equation~\eqref{Killing-type} with $\tau\equiv 0$, $f\equiv 0$, and $\xi(q)$ as a function of $q$~only, becomes $\partial_q L\cdot \xi(q)+ \partial_{\dot q}L\cdot\xi'(q)\dot q=0$, which is quadratic homogeneous in~$\dot q$: if we equate to zero the coefficients, we get the well-known Killing equations of Differential Geometry. The first integral~\eqref{firstintegral} simplifies to $-\partial_{\dot q} L\cdot \xi(q)=-A(q)\dot q\cdot \xi(q)=-\dot q \cdot A(q) \xi(q)$. 

Back to the general Killing-type equation~\eqref{Killing-type}, we will tacitly assume that $L\in C^3$ and that the following usual regularity condition is satisfied:
\begin{equation}\label{Legendre}
  \det g\ne0\qquad
  \text{where }g:=\partial^2_{\dot q,\dot q}L(t,q,\dot q).
\end{equation}
The notation $\partial^2_{\dot q,\dot q}L$ means the Hessian matrix of the second derivatives of~$L$ with respect to~$\dot q$. This will ensure that the Lagrange equation can indeed be put into normal form~\eqref{normalformLagrange}, and that there is existence and uniqueness of the solutions to the Cauchy problems.

The following result was basically found by Lutzky~\cite[formula~(19)]{Lutzky}, who uses it to argue that it is not useful to allow $f$ to depend on~$\dot q$.

\begin{Teorema}[General solution of the on-flow equation]\label{familyOfSolutionsWithLagrange}
Let $L(t,q,\dot q)$ be a Lagrangian function. Suppose that the Lagrange equation has the $C^1$ first integral $N(t,q,\dot q)$. Then a triple $(\tau,\xi,f)$ is a solution of the on-flow version of the Killing-type equation with~$N$ as associated first integral if and only if
\begin{equation}\label{familyFormula}
  f=\tau L+N+\partial_{\dot q}L\cdot(\xi-\tau\dot q).
\end{equation}
\end{Teorema}

\begin{proof}
Equation~\eqref{familyFormula} is simply a rearrangement of formula~\eqref{firstintegral}. If the triple $(\tau,\xi,f)$ is a solution associated with~$N$ then it must satisfy~\eqref{familyFormula}.

Conversely, suppose that the triple satisfies~\eqref{familyFormula} and let us check that it is an on-flow solution. Taking the time derivative of~$f=\tau L+N+ \partial_{\dot q} L\cdot(\xi-\tau\dot q)$ along a solution of Lagrange equation and replacing into Killing-type equation~\eqref{Killing-type}
\begin{multline}
  \tau\partial_t L
  +\partial_q L\cdot \xi
  +\partial_{\dot q}L\cdot
  \bigl(\dot \xi-\dot q\dot \tau\bigr)
  +L\dot \tau=\\
  =\dot f=\dot \tau L+\tau \dot L+0+
  \Bigl(\frac{d}{dt}\partial_{\dot q} L\Bigr)\cdot(\xi-\tau \dot q)+
  \partial_{\dot q} L\cdot(\dot\xi-\dot \tau \dot q-\tau \ddot q)
\end{multline}
Canceling out the common terms and using Lagrange equation we get
\begin{equation}
  \tau\partial_t L 
  +\partial_q L\cdot \xi
  =\tau\dot L+
  \Bigl(\frac{d}{dt}\partial_{\dot q} L\Bigr)\cdot(\xi-\tau \dot q)+
  \partial_{\dot q} L\cdot(-\tau \ddot q).
\end{equation}
Using Lagrange equation~\eqref{Lagrange} this becomes
\begin{equation}
  \tau\partial_t L 
  +\partial_q L\cdot \xi
  =\tau \dot L+
 \partial_{q} L\cdot(\xi-\tau \dot q)-
  \tau \partial_{\dot q} L\cdot\ddot q.
\end{equation}
Canceling out and rearranging we get
\begin{equation}
  \tau \bigl(\partial_t L+\partial_{q} L\cdot q+
  \partial_{\dot q} L\cdot\ddot q\bigr)
  =\tau\dot L,
\end{equation}
which is simply the chain rule.
\end{proof}

From formula~\eqref{familyFormula} we can express $\tau $ as a function of $f,\xi$, at least when $L-\partial_{\dot q}L\cdot\dot q\ne0$. Of particular interest are the solutions with $f=0$:

\begin{Corollario}[Simplest on-flow solution with $f=0$]\label{existenceOnFlow}
Let $L$ be a Lagrangian function. Suppose that the Lagrange equation has the $C^1$ first integral $N$. Then the on-flow version of the Killing-type equation~\eqref{Killing-type} is satisfied by the triple
\begin{equation}\label{solutionTriple}
  \tau =-\frac{N}{L},\qquad
  \xi=-\frac{N}{L}\dot q,\qquad
  f\equiv 0.
\end{equation}
The corresponding first integral~\eqref{firstintegral} is precisely~$N$.
\end{Corollario}

Around points where $L=0$ we can take $\tau =-N/(L+c)$, $\xi=- N/ (L+c)$, for a constant $c\ne0$.

\begin{Corollario}[More general on-flow solution with $f=0$]\label{existenceOnFlow2}
Let $L(t,q,\dot q)$ be a Lagrangian function. Suppose that the Lagrange equation has the $C^1$ first integral $N(t,q,\dot q)$. Let $R(t,q,\dot q)$ be an arbitrary smooth function with values in~$\R^n$. Then the on-flow version of the Killing-type equation~\eqref{Killing-type} is satisfied by the triple
\begin{equation}\label{solutionTriple2}
  \tau (t,q,\dot q)=
  -\frac{N+\partial_{\dot q}L\cdot R}{L},\qquad
  \xi(t,q,\dot q)=R-\dot q
  \frac{N+\partial_{\dot q}L\cdot R}{L},\qquad
  f\equiv 0.
\end{equation}
The corresponding first integral~\eqref{firstintegral} is precisely~$N$.
\end{Corollario}

The solutions of the strong Killing equation were given indirectly by Sarlet and Cantrijn~\cite[Th.~6.1]{SarletCantrijn} as a consequence of a result on $d\theta$-symmetries. Here we give a direct formulation and proof.

\begin{Teorema}[General solution of the strong equation]\label{familyOfSolutionsWithoutLagrange}
Let $L(t,q,\dot q)$ be a Lagrangian function. Suppose that the Lagrange equation has the $C^2$ first integral $N(t,q,\dot q)$. Then a triple $(\tau ,\xi,f)$ is a solution of the strong version of the Killing-type equation with~$N$ as associated first integral if and only if
\begin{gather}\label{strongSolutionCondition}
  \xi= \tau \dot q-g^{-1}\partial_{\dot q}N,\\
  f=\tau L+N-\partial_{\dot q}L\cdot g^{-1}\partial_{\dot q}N,
  \label{strongSolutionGauge}
\end{gather}
where $g=\partial^2_{\dot q,\dot q}L$ is the Hessian matrix as in~\eqref{Legendre}.
\end{Teorema}

\begin{proof}
Let us establish some formulas first. The function~$\Lambda$ appearing in the normal form of Lagrange equation~\eqref{normalformLagrange} can be made explicit this way:
\begin{equation}
  \Lambda\equiv g^{-1}\bigl(\partial_{q}L-\partial^2_{\dot q,t}L-
  \partial^2_{\dot q,q}L\;\dot q\bigr).
\end{equation}
For any smooth $q(t)$, not necessarily a Lagrangian motion, the following relations holds:
\begin{gather}
  \partial_qL-\frac{d}{dt}\partial_{\dot q}L=g(\Lambda-\ddot q),
  \label{differenceOfLagrangeSides}\\
  \dot L=\partial_tL+\partial_qL\cdot\dot q
  +\partial_{\dot q}L\cdot\ddot q\label{lDot}
\end{gather}
Since $N$ is a first integral, again for any smooth $q(t)$ we have
\begin{equation}\label{nDot}\begin{split}
  \dot N={}&\partial_tN+\partial_qN\cdot\dot q+
  \partial_{\dot q}N\cdot\ddot q=\\
  ={}&\underbrace{\partial_tN+\partial_qN\cdot\dot q+
  \partial_{\dot q}N\cdot \Lambda}_{=0}+
  \partial_{\dot q}N\cdot(\ddot q-\Lambda)=\\
  ={}&\partial_{\dot q}N\cdot(\ddot q-\Lambda).
  \end{split}
\end{equation}

A solution of the strong form with $N$ as associated first integral must in particular satisfy relation~\eqref{familyFormula}. Suppose that the triple $(\tau ,\xi,f)$ satisfies~\eqref{familyFormula} and let us impose that it solves the strong version of Killing equation. The total time derivative of~$f$ along a generic smooth $q(t)$ is:
\begin{align}
  \dot f={}&
  \frac{d}{dt}\bigl(\tau L+N+\partial_{\dot q}L\cdot(\xi-\tau \dot q)
  \bigr)=\\
  ={}&\dot \tau L+\tau \dot L+\dot N+
  \Bigl(\frac{d}{dt}\partial_{\dot q}L\Bigr)\cdot(\xi-\tau \dot q)+
  \partial_{\dot q}L\cdot
  (\dot\xi-\dot \tau \dot q-\tau \ddot q).
\end{align}
Equating this with the left-hand side of the strong Killing-like equation we get
\begin{multline}
  \tau\partial_t L 
  +\partial_q L\cdot \xi
  +\partial_{\dot q}L\cdot
  \bigl(\dot \xi-\dot q\dot\tau \bigr)
  +L\tau =\\
  =\dot \tau L+\tau \dot L+\dot N+
  \Bigl(\frac{d}{dt}\partial_{\dot q}L\Bigr)\cdot(\xi-\tau \dot q)+
  \partial_{\dot q}L\cdot
  (\dot\xi-\dot \tau \dot q-\tau \ddot q)
\end{multline}
which simplifies immediately to
\begin{equation}
  \tau\partial_t L 
  +\partial_q L\cdot \xi
  =\tau \dot L+\dot N+
  \Bigl(\frac{d}{dt}\partial_{\dot q}L\Bigr)\cdot(\xi-\tau \dot q)-
  \tau \partial_{\dot q}L\cdot
  \ddot q.
\end{equation}
Using~\eqref{lDot} to replace~$\dot L$ it becomes
\begin{equation}
  \tau\partial_t L 
  +\partial_q L\cdot \xi
  =(\partial_tL+\partial_qL\cdot\dot q)\tau +\dot N+
  \Bigl(\frac{d}{dt}\partial_{\dot q}L\Bigr)\cdot(\xi-\tau \dot q),
\end{equation}
which further simplifies to
\begin{equation}
  \partial_q L\cdot \xi
  =\tau \partial_qL\cdot\dot q+\dot N+
  \Bigl(\frac{d}{dt}\partial_{\dot q}L\Bigr)\cdot(\xi-\tau \dot q),
\end{equation}
which can be rearranged to
\begin{equation}\label{productNdot}
  \Bigl(\partial_q L-\frac{d}{dt}\partial_{\dot q}L\Bigr)
  \cdot(\xi-\tau \dot q)=
  \dot N.
\end{equation}
Using~\eqref{differenceOfLagrangeSides} and~\eqref{nDot}, formula~\eqref{productNdot} becomes
\begin{equation}
  \bigl(g(g-\ddot q)\bigr)\cdot
  (\xi-\tau \dot q)=\partial_{\dot q}N(\ddot q-\Lambda)=
  -\partial_{\dot q}N\cdot(\Lambda-\ddot q).
\end{equation}
Since the hessian matrix~$g$ is symmetric, this becomes
\begin{equation}
  \bigl(g(\xi-\tau \dot q)\bigr)\cdot
  (\Lambda-\ddot q)=-\partial_{\dot q}N\cdot(\Lambda-\ddot q).
\end{equation} 
Finally, since $\ddot q$ is arbitrary, we conclude that
\begin{equation}
  g(\xi-\tau \dot q)=-\partial_{\dot q}N,
\end{equation}
which is equivalent to~\eqref{strongSolutionCondition}. Equation~\eqref{strongSolutionGauge} is simply a consequence of~\eqref{strongSolutionCondition} and~\eqref{familyFormula}.
\end{proof}

A direct proof of the ``if'' part of Theorem~\ref{familyOfSolutionsWithoutLagrange} can be found in a paper by Boccaletti and Pucacco~\cite[Sec.~2]{BoccalettiPucacco}.

If we know a solution to the Killing-type equation, either on-flow or strong, we can easily generate infinitely many others, parameterized by an arbitrary function:

\begin{Corollario}[Multiplicity for both on-flow and strong equation]\label{multiplicityBoth}
Let $L(t,q,\dot q)$ be a Lagrangian function. Suppose that the triple $\bigl(\tau (t,q,\dot q), \xi(t,q,\dot q),\allowbreak f(t,q,\dot q))$ satisfies the Killing-type equation~\eqref{Killing-type} in either the strong or the on-flow version. Take an arbitrary smooth function $h(t,q,\dot q)$. Then also the following triple
\begin{equation}\label{equivalenttriples}
  \tilde \tau =\tau +\frac{h-f}{L},\qquad
  \tilde\xi=\xi+\dot q\,\frac{h-f}{L},\qquad
  \tilde f=h 
\end{equation}
satisfies the Killing-type equation of the same form. The corresponding first integral \eqref{firstintegral} is the same.
\end{Corollario}

\begin{proof}
If we assume that any of the equations~\eqref{familyFormula}, \eqref{strongSolutionGauge} and~\eqref{strongSolutionCondition} holds for the triple $(\tau ,\xi,f)$, a simple replacement shows that the equation holds also for $(\tilde \tau ,\tilde\xi,\tilde f)$.
\end{proof}

Within the family of solutions given by Theorem~\ref{multiplicityBoth} there is always one with trivial (i.e., zero) time change and another one with trivial boundary term. This simple fact was already established in a more general setting (including, for example, nonlocal constants of motion) and different notations by the authors~\cite[Theorem~10]{GorniZampieri}.

\begin{Corollario}[Trivializing either time-change or gauge]\label{trivializationCorollary}
Let $L(t,q,\dot q)$ be a Lagrangian function. Suppose that the triple $\bigl(\tau (t,q,\dot q), \xi(t,q,\dot q),\allowbreak f(t,q,\dot q))$ satisfies the Killing-type equation~\eqref{Killing-type} in either the strong or the on-flow form. Then also the following two triples are solutions:
\begin{equation}\label{trivializationFormulas}
  (0,\;\xi-\dot q\tau ,\;f-L\tau ),\qquad 
  \Bigl(\tau -\frac{f}{L},\;\xi-\dot q\,\frac{f}{L},\;0\Bigr).
\end{equation}
The corresponding first integrals are the same.
\end{Corollario}

\begin{proof}
Simply take either $h=f-L\tau $ or $h=f$ in Corollary~\ref{multiplicityBoth}.
\end{proof}

\section{Strong solutions independent of $\dot q$}
\label{strongIndependentOfDotQ}

Given $\xi$ and~$\tau $ that do not depend on~$\dot q$ there is a simple necessary condition for them to be part of a solution triple $(\tau ,\xi,f)$ of the strong form of Killing-like equation, regardless of~$N$.

\begin{Proposizione}\label{independenceOfVelocity}
Suppose that $(\tau ,\xi,f)$ solves the strong form of Killing equation, and that $\xi(t,q)$ and~$\tau (t,q)$ do not depend on~$\dot q$. Then $f$ does not depend on~$\dot q$ either. Moreover, the left-hand side of the Killing-like equation
\begin{equation}\label{leftHandSide}
  \tau\partial_t L 
  +\partial_q L\cdot \xi
  +\partial_{\dot q}L\cdot
  \bigl(\dot \xi-\dot q\dot\tau \bigr)
  +L\dot \tau 
\end{equation} 
after replacing with the given $L(t,q,\dot q),\xi(t,q),\tau (t,q)$, depends linearly on~$\dot q$.
\end{Proposizione}

\begin{proof}
Starting from formula~\eqref{familyFormula}, which holds in the strong case too,
\begin{equation}
  f=\tau L+N+\partial_{\dot q}L\cdot(\xi-\tau \dot q)
\end{equation}
and taking the gradient with respect to~$\dot q$ we get
\begin{equation}\begin{split}
  \partial_{\dot q}f={}&\partial_{\dot q}N+
  \partial_{\dot q}\bigl(
  \tau L+\partial_{\dot q}L\cdot (\xi -\tau \dot q)\bigr)=\\
  ={}&
  -g(\xi -\tau \dot q)+
  \partial_{\dot q}\bigl(
  \tau L+\partial_{\dot q}L\cdot (\xi -\tau \dot q)\bigr),
  \end{split}
\end{equation}
where we have used the replacement $\partial_{\dot q}N=-g(\xi -\tau \dot q)$, which is a rearrangement of~\eqref{strongSolutionCondition}.
Using now the assumption that $\xi,\tau $ do not depend on~$\dot q$ we can carry on the calculation
\begin{equation}\begin{split}
  \partial_{\dot q}f={}&
  -g(\xi -\tau \dot q)+
  \tau \partial_{\dot q}L+
  g (\xi -\tau \dot q)+
  \partial_{\dot q}L(-\tau )
  \equiv0.
  \end{split}
\end{equation}
We deduce that $f$ does not depend on~$\dot q$ either. Hence $\dot f$ is linear in~$\dot q$. We conclude that expression~\eqref{leftHandSide}, which is identically equal to~$\dot f$, must be linear in~$\dot q$ too.
\end{proof}

Only some first integrals $N$ can be deduced from a triple $(\tau (t,q),\xi(t,q),\allowbreak f(t,q))$ independent of~$\dot q$.

\begin{Proposizione}\label{conditionOnTXiForIndependenceOfVelocity}
A first integral $N(t,q,\dot q)$ can be deduced from a triple $(\tau ,\xi, f)$ that does not depend on~$\dot q$ if and only if $g^{-1} \partial_{ \dot q}N=a(t,q)+b(t,q)\dot q$, where $a(t,q)$ is vector-valued and $b(t,q)$ is scalar-valued.
\end{Proposizione}

\begin{proof}
If $N$ can be deduced from $\tau (t,q),\xi(t,q),f(t,q)$, then from Theorem~\ref{familyOfSolutionsWithoutLagrange}, formula~\eqref{strongSolutionCondition}, $g^{-1}\partial_{\dot q}N= -\xi(t,q)+\tau (t,q)\dot q$. Conversely, if $g^{-1} \partial_{ \dot q}N=a(t,q)+b(t,q)\dot q$, we can choose $\xi=-a$, $\tau =b$ and $f$ given by equation~\eqref{strongSolutionGauge}, so that the triple $(\tau ,\xi, f)$ is a solution of Killing-like equation in the strong sense (Theorem~\ref{familyOfSolutionsWithoutLagrange}). Finally, $f$~does not depend on~$\dot q$ because of Proposition~\ref{independenceOfVelocity}.
\end{proof}

\section{Examples}\label{examplesSection}

\subsection{The free particle}
\label{FreeParticle}

In one of his papers~\cite{Leach2}, Leach argues that Lie point symmetries that are usually called ``nonnoetherian'' are indeed fully Noetherian, provided that we switch from the obvious Lagrangian to some other Lagrangian which retains the same equations of motion. The point is illustrated with the example of the free particle in one dimension: the equation of motion is $\ddot q=0$, whose Lie symmetries $\xi\partial_q +\tau \partial_t$ are an 8-dimensional space (Table~\ref{simmetrie}). If we examine these symmetries in the Noetherian sense together with the ``natural'' Lagrangian~$L=\dot q^2/2$, we see that only five of them can be completed to a Noetherian triple $(\tau ,\xi,f)$. Let us see what we can say about the three remaining ``nonnoetherian'' symmetries from the point of view of the strong and on-flow solutions, without resorting to a different Lagrangian.

\begin{table}
\begin{equation*}
\begin{array}{lccc}
\text{Lie symmetry}&\text{Lie 1st integral}&
f&
\text{Noether 1st int.}\\\hline
\mathstrut\Gamma_1=\partial_q & \dot q & 0 & -\dot q\\
\Gamma_2=t\partial_q & t\dot q-q & q & q-t\dot q\\
\Gamma_3=\partial_t & \dot q & 0 & \dot q^2/2\\
\Gamma_4=2t\partial_t+q\partial_q & (t\dot q-q)\dot q & 0 &
   (t\dot q-q)\dot q\\
\Gamma_5=t^2\partial_t+tq\partial_q & t\dot q-q & q^2/2 &
   (q-t\dot q)^2/2\\[3pt]
\Gamma_6=q\partial_q & (t\dot q-q)/\dot q\\
\Gamma_7=q\partial_t & (t\dot q-q)/\dot q\\
\Gamma_8=qt\partial_t+q^2\partial_q & (t\dot q-q)/\dot q
\end{array}
\end{equation*}
\begin{center}
\caption{Lie and Noether symmetries of the free particle}
\label{simmetrie}
\end{center}
\end{table}

For the Lie symmetries $\Gamma_6,\Gamma_7,\Gamma_8$ the Killing-like equation becomes respectively
\begin{equation}
  \dot q^2=\dot f_6,\qquad
  -\frac{\dot q^3}{2}=\dot f_7,\qquad
  \frac{3q\dot q^2-t\dot q^3}{2}=\dot f_8.
\end{equation}
These equations have no solution in~$f$ in the strong sense, because the left-hand sides are not linear in~$\dot q$ (Proposition~\ref{conditionOnTXiForIndependenceOfVelocity}).

As for the on-flow version of the Killing-like equation, given any arbitrary couple $(\tau ,\xi)$ and a first integral~$N$, formula~\eqref{familyFormula} of Theorem~\ref{familyOfSolutionsWithLagrange} immediately gives a boundary term~$f$ that completes to a solution triple. Specifically:
\begin{gather}
  \text{for }\Gamma_6\qquad
  f_6=q\dot q+N
  \\
  \text{for }\Gamma_7\qquad
  f_7=-\frac{q\dot q^2}{2}+N
  \\
  \text{for }\Gamma_8\qquad
  f_8=\frac{1}{2}(2q-t\dot q)q\dot q+N.
\end{gather}
Of course, these boundary terms depend on~$\dot q$. With solutions in the on-flow solutions we can recover all first integral of the system, i.e., all function of the form $k(\dot q,q-t\dot q)$. Using Proposition~\ref{conditionOnTXiForIndependenceOfVelocity}, since $g^{-1}=1$, we can say that with solutions $(\tau (t,q),\xi(t,q),f(t,q))$ in the strong sense we cannot obtain first integrals that are not quadratic in~$\dot q$, for example $(q-t\dot q)^3$.

\subsection{Superintegrable systems related to isochrony}
\label{superintegrableSubSection}

The authors are familiar with with the following system of two scalar differential equations
\begin{equation}\label{isoch}
  \ddot x=-G(x),\quad
  \ddot y=-G'(x)y,
\end{equation}
which are the Lagrange equations of the Lagrangian
\begin{equation}\label{LagrangianoIsoch}
  L(t,q,\dot q)=
  \dot x\,\dot y-G(x)y,
  \qquad\text{where }q=\binom{x}{y},\ 
  \dot q=\binom{\dot x}{\dot y}
\end{equation}
This system is rich with first integrals. One is $N_1=\dot x\dot y+G(x)y$. Since $g^{-1}\partial_{\dot q}N_1\allowbreak=\dot q$, following Proposition~\ref{conditionOnTXiForIndependenceOfVelocity} we can deduce $N_1$ from the following triple independent of~$\dot x,\dot y$:
\begin{equation}
  \tau =1,\quad
  \xi=0,\quad
  f=0,
\end{equation}
which is a solution in the strong sense.

Another first integral is $N_2=\dot x^2/2+\int G(x)dx$. Since $g^{-1}\partial_{\dot q}N_2=(\begin{smallmatrix}0&0\\ -1&0 \end{smallmatrix})\dot q$ and because of Proposition~\ref{conditionOnTXiForIndependenceOfVelocity}, to deduce $N_2$ we must accept dependence on~$\dot q$. According to Theorem~\ref{familyOfSolutionsWithoutLagrange}, all solutions in the strong sense are given by an arbitrary~$\tau $ and
\begin{gather}
  \xi=\tau \dot q-g^{-1}\partial_{\dot q}N_2=
  \binom{\tau \dot x}{\tau \dot y-\dot x},\\
  f=\tau L+N_2-\partial_{\dot q}g^{-1}\partial_{\dot q}N_2=
  \tau \bigl(\dot x\dot y-G(x)y\bigr)-\frac{1}{2}\dot x^2+\int G(x)dx
\end{gather}

The special feature of the system~\eqref{isoch} is that for some classes of function~$G$ the system has a third, independent, first integral. One way to detect some of these superintegrable system is by making a plausible ansatz on the triple $(\tau ,\xi,f)$ and solving the Killing-like equation for $G$ as an additional unknown function. We think it is preferable to use the on-flow version of the equation, simply because it has so many more solution, heightening the chances that the ansatz may catch one.

Our starting ansatz is
\begin{equation}\label{ansatzH}
  \tau \equiv0,\qquad
  \xi=\binom{h(x,\dot x)}{0}.
\end{equation}
In keeping with the on-flow version, the first and second total time derivatives of~$h$ will take the Lagrange equations into account:
\begin{equation}
  \dot h=\dot x\partial_x h+\ddot x\partial_{\dot x}h=
  \dot x\partial_x h-G(x)\partial_{\dot x}h,\qquad
  \ddot h=\dot x\partial_x \dot h-G(x)\partial_{\dot x}\dot h
\end{equation}
The Killing-like equation becomes
\begin{equation}
  -yG'(x)h+\dot y\dot h=\dot f.
\end{equation}
With the further ansatz that
\begin{equation}\label{ansatzF}
  f=y\dot h
\end{equation}
the equation is
\begin{equation}
  -yG'(x)h+\dot y\dot h=\dot y\dot h+y\ddot h
\end{equation}
which simplifies to
\begin{equation}\label{killingWithH}
  -G'(x)h=\ddot h.
\end{equation}
We can make a third ansatz by setting $h$ to be a polynomial in~$\dot x$ of the form $h=\alpha(x)+\beta(x)\dot x^2$. Replacing into~\eqref{killingWithH} we obtain a polynomial of degree~4 in~$\dot x$ equated to~0:
\begin{multline}\label{polyInXdot}
  \beta ''(x)\dot x^4 +
  \bigl(\alpha''(x)-\beta(x)G'(x)-5 G(x)\beta'(x)\bigr)\dot x^2+{}\\
  +2 G(x)^2\beta(x)-G(x)\alpha '(x)
  +\alpha (x) G'(x)=0.
\end{multline}
The coefficient of~$\dot x^4$ is $\beta''(x)$, which must be~0. Let us simply take $\beta(x)\equiv x$. Equating the coefficient of~$\dot x^0$ in~\eqref{polyInXdot} to~0 we get
\begin{equation}
  \alpha(x) G'(x)-G(x)\alpha '(x)+2 x G(x)^2=0,
\end{equation}
which can be solved for~$\alpha$ as $\alpha(x)=(c+x^2)G(x)$. Replacing into the coefficient of~$\dot x^2$ we get the second order linear equation in~$G$
\begin{equation}\label{equationInG}
  (c+x^2) G''(x)+3 x G'(x)-3 G(x)=0,
\end{equation}
whose linear space of real solutions around $x=0$ is generated by $G(x)=x$ and by
\begin{equation}
  \frac{1}{x^3}\quad\text{if }c=0,\qquad
  \frac{c+2x^2}{\sqrt{c+x^2}}\quad\text{if }c>0,
  \qquad
  \frac{-c-2x^2}{\sqrt{-c-x^2}}\quad\text{if }c<0,
\end{equation}
If we take any $G$ in this space, and set $h=(c+x^2)G(x)+x\dot x^2$, the triple given by equations~\eqref{ansatzH} and~\eqref{ansatzF} is an on-flow solution to Killing-like equation, and the associated first integral is
\begin{equation}\begin{split}
  N_3={}&f-L\tau -\partial_{\dot q}L\cdot\bigl(\xi-\tau \dot q\bigr)=\\
  ={}&y\dot h-\partial_{\dot q}L\cdot\xi=\\
  ={}&(c+x^2)G'(x)\dot xy-(c+x^2) G(x)\dot y-x
   \dot x^2 \dot y+\dot x^3 y.
  \end{split}
\end{equation}
It can be verified that the three first integrals $N_1,N_2,N_3$ are functionally independent. The triple $(\tau ,\xi,f)$ that we have found is not a solution in the strong sense, since $\xi-(\tau \dot q-g^{-1}\partial_{\dot q}N_3)=(0, (c+x^2) G'(x)y+(3 \dot{x} y-2 x\dot{y})\dot{x} )$ does not vanish identically.

Now that we know the expression of the first integral~$N_3$ we can construct the solution triples $(\tau ,\Xi,f)$ of the Killing-like equation in the strong sense:
\begin{gather}
  \tau =T,\qquad
  \Xi=\binom{(c+x^2) G(x)+
  (T +x \dot{x})\dot{x}}{
  -y(c+x^2) G'(x)-3 \dot{x}^2 y+2 x \dot{x} \dot{y}
  +T \dot{y}},\\
  f=(T +2 x \dot{x})\dot{x} \dot{y} 
  -2 \dot{x}^3 y-G(x)T y.
\end{gather}

Summing up, we have used the on-flow version of the Killing-like equation to detect a class of superintegrable systems:

\begin{Proposizione}
The Lagrangian system given by equations~\eqref{isoch} and~\eqref{LagrangianoIsoch} is superintegrable whenever the function~$G$ satisfies equation~\eqref{equationInG}.
\end{Proposizione}

This class of systems with the parameter $c>0$ overlaps with the one that was found by the second author~\cite[Sec.~5]{Zampieri} using a totally different line of reasoning. The on-flow method that we have illustrated here was pushed further (albeit with a different language) by the two authors~\cite[Sec.~13]{GorniZampieri}, using with the more general ansatz $h=\alpha(x)+\beta(x)\dot x^2+\gamma(x)\dot x^4$. We expect that a larger class of superintegrable systems can readily be found by increasing the degree of~$h$ with respect to~$\dot x$.

\subsection{The Laplace-Runge-Lenz vector for Kepler's problem}\label{LRLSubSection}

Consider the Lagrangian function and Lagrange equation of Kepler's problem in dimension~3
\begin{gather}\label{LKepler}
  L(t,\vec r, \vec v)=
  \frac{1}{2}  \lVert\vec v\rVert^2
  +\frac{\mu}{\lVert \vec r\rVert},\quad
  \vec r\in\R^3\setminus\{0\},
  \\
  \label{LagrangeKepler}
  \ddot{\vec r}=-\frac{\mu}{\lVert \vec r\rVert^3}\,\vec r\,.
\end{gather}
Here we depart from the $q,\dot q$ notation and use $\vec r,\vec v$ instead, as done in common introductory mechanics textbooks. The vector product~``$\times$'' for 3-dimensional vectors will allow more compact formulas than what we get in the otherwise equivalent 2-dimensional treatment we gave in an earlier paper~\cite{GorniZampieri}.

Besides energy and angular momentum, the Kepler system has the LRL vector first integral
\begin{equation}
  \vec A:= \vec v\times
  (\vec r\times  \vec v)-
  \frac{\mu}{\lVert \vec r\rVert}\,\vec r.
\end{equation}
Fix an arbitrary vector $\vec u\in\R^3$ and consider the scalar first integral
\begin{equation}
  N:=-\vec u\cdot \vec A.
\end{equation}
 If we check the condition of Proposition~\ref{conditionOnTXiForIndependenceOfVelocity} we see that $N$ cannot be obtained from a triple $(\tau,\xi,f)$ which is independent of~$\vec v$. Let us see what we can do with either on-flow or strong solutions involving~$\vec v$.

Theorem~\ref{familyOfSolutionsWithLagrange} gives us so many on-flow solutions that we may be choosy and aim for subjectively simple, elegant formulas. One that is simple enough is
\begin{equation}
  \tau _0=\frac{\vec u\cdot \vec v\times
  (\vec r\times \vec v)}{L},\qquad
  {\vec \xi}_0=\frac{\vec u\cdot \vec v\times(\vec r\times \vec v)}{L}
  \vec v,\qquad
  f_0=\frac{\mu}{\lVert\vec r\rVert}\vec r\cdot\vec u.
\end{equation}
Levy-Leblond~\cite{Leblond} proposed the following one, without explanation as to how he came up with the formula:
\begin{gather}
  \tau _L=0,\qquad
  {\vec\xi}_L=-\frac{1}{2}\partial_{\vec v}N=
  (\vec r\cdot \vec u)\vec v
  -\frac{1}{2}(\vec v\cdot \vec u)\vec r
  -\frac{1}{2}(\vec v\cdot \vec r)\vec u,\\
  f_L=\tau_L L-N+\partial_{\dot q}L\cdot(\xi_L-\tau_L\dot q)=
  \frac{\mu}{\lVert\vec r\lVert}\vec r\cdot\vec u=f_0,
\end{gather}

To find different solutions with trivial first order time variation $\tau \equiv 0$, we write the Killing-type equation~\eqref{Killing-type} within the current setting:
\begin{equation}\label{KillingWithZeroT}
  \partial_{\vec r} L\cdot \vec\xi
  +\partial_{\vec v}L\cdot
  \dot {\vec\xi}=
  \dot f,
\end{equation}
and we impose that the first order space variation $\vec \xi$ be such that
\begin{equation}\label{imposition}
 \partial_{\vec v}L\cdot \vec \xi=\vec u\cdot \vec v\times
  (\vec r\times  \vec v).
\end{equation}
Our favourite way to satisfy this condition is
\begin{equation}
  {\vec \xi}_Z=
  (\vec r\times  \vec v)\times\vec u=
  (\vec r\cdot \vec u)\vec v-(\vec v\cdot \vec u)\vec r.
\end{equation}
This choice is not the only possible: $\vec\xi_L$ satisfies the same condition:
\begin{equation}
  \partial_{\vec v}L\cdot\vec \xi_{L}
  =\vec v\cdot\vec \xi_{L}=
  \vec u\cdot \vec v\times
  (\vec r\times  \vec v)
  =\lVert\vec v\rVert^2\vec u\cdot\vec r
  -(\vec v\cdot \vec r)(\vec v\cdot \vec u).
\end{equation}
Using Lagrange equation \eqref{LagrangeKepler} we have  
\begin{equation}
  {\dot {\vec\xi}}_Z=
  \bigl(\partial_{\vec r}{\vec \xi}_Z\bigr) \vec v
  +\partial_{\vec v}{\vec \xi}_Z\,
  \Bigl(-\frac{\mu}{\lVert \vec r\rVert^3}\, \vec r\Bigr)
  =\vec 0.
\end{equation}
Using equation~\eqref{imposition}, the left-hand side of equation~\eqref{KillingWithZeroT} becomes
\begin{align*}
  \partial_{\vec r} L\cdot {\vec \xi}_Z
  +\partial_{\vec {}v}L\cdot
  {\dot {\vec \xi}}_Z={}&
  \partial_{\vec r} L\cdot {\vec \xi}_Z
  +\partial_{\vec v}L\cdot \vec 0=
  -\frac{\mu}{\lVert \vec r\rVert^3}\vec r\cdot
  (\vec r\times  \vec v)\times\vec u=\\
  ={}&-\frac{\mu}{\lVert \vec r\rVert^3}\vec r\times
  (\vec r\times  \vec v)\cdot\vec u=\\
  ={}&-
  \frac{\mu}{\lVert \vec r\rVert^3}\,\vec u\cdot
  \bigl(\vec r(\vec r\cdot \vec v)-
  \vec v\Vert \vec r\rVert^2\bigr)=
  \vec v\cdot\partial_{\vec r}
  \Bigl(\frac{\mu}{\lVert \vec r\rVert}\vec r\cdot\vec u\Bigr)=\\
  ={}&\frac{d}{dt}
  \Bigl(\frac{\mu}{\lVert \vec r\rVert}\vec r\cdot\vec u\Bigr)=
  \dot f_0.
\end{align*}
We can complete the solution triple as follows:
\begin{equation}\label{LRLgauge}
  \tau _Z\equiv0,\qquad
  {\vec \xi}_Z=
  (\vec r\times  \vec v)\times\vec u,\qquad
  f_Z=f_0=\frac{\mu}{\lVert \vec r\rVert}\,\vec r\cdot\vec u.
\end{equation}
The resulting first integral of formula~\eqref{firstintegral} is what we expected:
\begin{align*}
  f_Z-\partial_{\dot \vec v}L\cdot{\vec \xi}_Z={}&
  \frac{\mu}{\lVert \vec r\rVert}\vec r\cdot\vec u
  -\vec v\cdot (\vec r\times  \vec v)\times\vec u=\\
  ={}&-\Bigl(\vec v\times (\vec r\times  \vec v)
  -\frac{\mu}{\lVert \vec r\rVert}\vec r\Bigr)\cdot\vec u=
  N.
\end{align*}
The triple~\eqref{LRLgauge} belongs to the family of Theorem~\ref{familyOfSolutionsWithLagrange}, as can be checked by direct computation.

Corollary~\ref{multiplicityBoth}, formula~\eqref{equivalenttriples}, applied to the triple~\eqref{LRLgauge}, gives a whole family of solution triples, depending on an arbitrary function~$h(t,\vec r,\vec v)$:
\begin{equation}\label{familyForLRL}
  \tau =\frac{1}{L}\Bigl(h-\frac{\mu}{\lVert \vec r\rVert}
  \vec u\cdot\vec r\Bigr),\quad
  \vec \Xi= (\vec r\times  \vec v)\times\vec u
  +\frac{1}{L}\Bigl(h-\frac{\mu}{\lVert \vec r\rVert}
  \vec u\cdot\vec r\Bigr)\vec v,\quad
  f=h.
\end{equation}
As in Corollary~\ref{trivializationCorollary}, the choice $h\equiv0$ will trivialize the boundary term.

Using a computer algebra system the reader can directly check all these solutions to the Killing-type equation, independently of the theorems in Section~\ref{KillingSection}. For example here is some simple code written for Wolfram \emph{Mathematica} that implements the solution triple~\eqref{familyForLRL} and then checks that the on-flow Killing-type equation is satisfied and that the first integral is the LRL vector:

\begin{verbatim}
(*Defining the variables*)
r = {r1, r2, r3};
v = {v1, v2, v3};
L = v.v/2 + mu/Sqrt[r.r];
A = Cross[v, Cross[r, v]] - mu*r/Sqrt[r.r];
u = {u1, u2, u3};
arbitrary = h[t, r1, r2, r3, v1, v2, v3];
f0 = mu*(u.r)/Sqrt[r.r];
T = (arbitrary - f0)/L;
Xi = Cross[Cross[r, v], u] + T*v;
f = arbitrary;
(*the time dot derivatives are on-flow*)
rDotDot = -mu*(u.r)*r/(r.r)^(3/2);
Tdot = D[T, t] + D[T, {r}].v + D[T, {v}].rDotDot;
XiDot = D[Xi, t] + D[Xi, {r}].v + D[Xi, {v}].rDotDot;
fDot = D[f, t] + D[f, {r}].v + D[f, {v}].rDotDot;

(*checking on-flow Killing-type equation*)
Simplify[
 D[L, t]*T + D[L, {r}].Xi + D[L, {v}].(XiDot - v*Tdot)
  + L*Tdot == fDot]

(*checking LRL vector as first integral*)
Simplify[
 f - L*T - D[L, {v}].(Xi - T*v) == -A.u]
\end{verbatim}

\noindent
Upon evaluation, the code gives \verb!True! and \verb!True! in an instant.

Solutions of the Killing-type equation in the strong sense are fewer, and there is less freedom to simplify formulas. The explicit triples given by Sarlet and Cantrijn~\cite[Sec.~6]{SarletCantrijn}, and by the authors~\cite[Sec.~12]{GorniZampieri} are for dimension~2. Boccaletti an Pucacco~\cite[Sec.~2.2]{Boccaletti} deduce their solution, also in dimension~2, by assuming $\tau\equiv0$ and $\xi$ to be a bilinear function of~$q,\dot q$ and then working out the coefficients. Here we contribute a triple written for dimension~3, where the vector cross product again leads to elegant formulas for the solution $(\tau ,\Xi,f)$, and where we incorporate the multiplicity Corollary~\ref{multiplicityBoth}:
\begin{gather}
  \vec b:=-\vec u \bigl(\vec r\cdot \vec v\bigr)
  -\vec r\bigl(\vec v\cdot \vec u\bigr)
  +\vec v\bigl(\vec u\cdot \vec r\bigr),\\
  \tau =\frac{1}{L}\biggl(h-\vec u\cdot
  \Bigl(\vec v\times \bigl(\vec r\times \vec v\bigr)
  +\frac{\mu}{\lVert \vec r\rVert}\, \vec r\Bigr)\biggr),\\
  \vec \Xi=\frac{1}{L}
  \Bigl(h\,\vec v+\frac12
  \vec v\times\bigl(\vec b\times\vec v\bigr)
  +\frac{\mu}{\lVert \vec r\rVert}\,\vec b\Bigr),\qquad
  f=h,
\end{gather}
where $\vec u\in\R^3$ is an arbitrary parameter vector, as in the previous section. The first integral associated to the triple through Noether's theorem~\eqref{firstintegral} is the same as before:
\begin{equation}
  L\tau +\partial_{\vec v}L\cdot
  \bigl(\vec \Xi-\tau  \vec v\bigr)=
  -\vec u\cdot\Bigl(\vec v\times
  \bigl(\vec r\times  \vec v\bigr)-
  \frac{\mu}{\lVert \vec r\rVert}\,\vec r\Bigr)=-\vec u\cdot\vec A.
\end{equation} 
Again we provide below some \emph{Mathematica} code that implements the solution triple and checks that it solves the Killing-type equation in the strong sense, and that it gives the Laplace-Runge-Lenz first integral.

\begin{verbatim}
(*Defining the variables*)
r = {r1, r2, r3};
v = {v1, v2, v3};
L = v.v/2 + mu/Sqrt[r.r];
A = Cross[v, Cross[r, v]] - mu*r/Sqrt[r.r];
u = {u1, u2, u3};
arbitrary = h[t, r1, r2, r3, v1, v2, v3];
T = (arbitrary - u.(Cross[v, Cross[r, v]] + mu*r/Sqrt[r.r]))/L;
b = -u*(r.v) - r*(u.v) + v*(r.u);
Xi = (arbitrary*v + Cross[v, Cross[b, v]]/2 + mu*b/Sqrt[r.r])/L;
f = arbitrary;
(*the time dot derivatives are generic, not on-flow*)
rDotDot = {a1, a2, a3};
Tdot = D[T, t] + D[T, {r}].v + D[T, {v}].rDotDot;
XiDot = D[Xi, t] + D[Xi, {r}].v + D[Xi, {v}].rDotDot;
fDot = D[f, t] + D[f, {r}].v + D[f, {v}].rDotDot;

(*checking Killing-type equation*)
Simplify[
 D[L, t]*T + D[L, {r}].Xi + D[L, {v}].(XiDot - v*Tdot)
  + L*Tdot == fDot]

(*checking LRL vector as first integral*)
Simplify[
 f - L*T - D[L, {v}].(Xi - T*v) == -A.u]\end{verbatim}

\noindent
The evaluation gives \verb!True!, as expected.


\section{Killing-like equation for a different time\\ change}
\label{differentTimeChangeSection}

Infinitesimal invariance up to boundary terms usually refers to the dependence of the quantity
\begin{equation}\label{actionWithInverseTimeChange}
  \int_{\bar t_1}^{\bar t_2}L\Bigl(\bar t,\bar q(\bar t),
  \frac{d\bar  q}{d\bar t}(\bar t)\Bigr)d\bar t
\end{equation}
with respect to $\varepsilon$, as in equation~\eqref{infinitesimalInvariance} of Section~\ref{introduction}, where again
\begin{equation}\label{timeAndSpaceChange}
  \bar t_\varepsilon(t)=t+\varepsilon
  \tau\bigl(t,q(t),\dot q(t)\bigr),\qquad
  \bar q_\varepsilon(t)=q+\varepsilon
  \xi\bigl(t,q(t),\dot q(t)\bigr).
\end{equation}
In an earlier paper on Noether's theorem~\cite{GorniZampieri} we proposed, among other things, a generalization of infinitesimal invariance that leads to nonlocal constants of motion, and also, more to the point here, that infinitesimal invariance with time change can be based on the following expression
\begin{equation}\label{actionWithNewTimeChange}
  \int_{\bar t_\varepsilon(t_1)}^{
  \bar t_\varepsilon(t_2)}L\Bigl(t,\bar q_\varepsilon(t),
  \frac{d\bar  q_\varepsilon}{dt}(t)\Bigr)dt.
\end{equation}
instead of~\eqref{actionWithInverseTimeChange} \cite[Sec.~4]{GorniZampieri}. What is different is that the time derivative of $\bar q_\varepsilon$ and the integration in~\eqref{actionWithNewTimeChange} is made by respect to the original time~$t$, whilst in Section~\ref{introduction} the derivative was made with respect to the transformed time~$\bar t$.

We will say that the transformation~\eqref{timeAndSpaceChange} leaves the action integral \emph{alternatively}-invariant up to boundary terms if a function $f(t,q,\dot q)$ exists, such that for all $t_1,t_2$ we have
\begin{multline}\label{infinitesimalInvarianceAlternative}
  \int_{\bar t_\varepsilon(t_1)}^{
  \bar t_\varepsilon(t_2)}L\Bigl(t,\bar q_\varepsilon(t),
  \frac{d\bar  q_\varepsilon}{dt}(t)\Bigr)dt+{}\\
  +\varepsilon\int_{t_1}^{t_2}
  \frac{df}{dt}\bigl(t,q(t),\dot q(t)\bigr)dt
  +O(\varepsilon^2)
  \quad\text{as }\varepsilon\to0.
\end{multline}
To translate this condition into a differential equation, we take the integral
\begin{equation}
  \int_{\bar t_\varepsilon(t_1)}^{
  \bar t_\varepsilon(t_2)}L\Bigl(t,\bar q_\varepsilon(t),
  \frac{d\bar  q_\varepsilon}{dt}(t)\Bigr)dt
\end{equation}
and replace the $t$ variable with $\bar t_\varepsilon(t)$. The integral becomes with fixed extrema $t_1,t_2$:
\begin{equation}
  \int_{t_1}^{t_2}L\bigl(\bar t_\varepsilon(t),
  \bar q_\varepsilon(\bar t_\varepsilon(t)),
  \dot{\bar q}(\bar t_\varepsilon(t))\bigr)
  \dot{\bar t}_\varepsilon(t)\,dt.
\end{equation}
The first-order expansion of this expression as $\varepsilon\to0$ is
\begin{multline}
  \int_{t_1}^{t_2}L\bigl(t,q(t),\dot q(t)\bigr)dt+{}\\
  +\varepsilon\int_{t_1}^{t_2}
  \bigl(\tau\partial_t L
  +\partial_q L\cdot (\xi+\tau\dot q)
  +\partial_{\dot q}L\cdot(\dot\xi+\tau\ddot q)
  +L\dot\tau\bigr)dt
  +O(\varepsilon^2).
\end{multline}
Hence the alternative invariance~\eqref{infinitesimalInvarianceAlternative} is equivalent to the following \emph{alternative} Killing-type equation for ODEs:
\begin{equation}\label{Killing-typeAlternative}
  \tau\partial_t L
  +\partial_q L\cdot (\xi+\tau\dot q)
  +\partial_{\dot q}L\cdot(\dot\xi+\tau\ddot q)
  +L\dot\tau=
  \dot f.
\end{equation}
The commonly made assumption that $\tau,\xi,f$ only depend on~$(t,q)$ eliminated~$\ddot q$ for the standard Killing-type equation~\eqref{Killing-type}, collapsing the on-flow and the strong interpretation. The alternative equation~\eqref{Killing-typeAlternative} does not lend itself to this simplification, so that we are forced to take position on how to understand~$\ddot q$: either as an independent $n$-dimensional variable (strong form), or as a shorthand for the $\Lambda(t,q,\dot q)$ of the Lagrange equation~\eqref{normalformLagrange} (on-flow form).

If a triple $(\tau,\xi,f)$ solves the alternative equation~\eqref{Killing-typeAlternative} in either sense, then the following function is a first integral for the Lagrangian system:
\begin{equation}\label{firstintegralAlternative}
  N=f-L\tau-\partial_{\dot q}L\cdot\xi.
\end{equation}
The expression is different from the corresponding formula~\eqref{firstintegral} for the standard Killing-like equation. Compare however our formula~\eqref{firstintegralAlternative} with  formula~(I.11) from Candotti, Palmieri and Vitale~\cite{Candotti}.

There is a simple correspondence between the solutions to the standard and the alternative Killing-type equations, either in the strong or in the on-flow interpretation: if $(\tau,\xi,f)$ solves the alternative version~\eqref{Killing-typeAlternative} then $(\tau,\xi+\tau\dot q,f)$ solves the standard~\eqref{Killing-type}. Conversely, if $(\tau,\xi,f)$ solves the standard~\eqref{Killing-type} then $(\tau,\xi-\tau\dot q,f)$ solves the alternative~\eqref{Killing-typeAlternative}. This is basically Theorem~8 of the previous paper~\cite{GorniZampieri}, except that the ``alternative'' tag is used in the opposite sense.

Given a first integral~$N$, the associated general solution for the strong interpretation of the standard Killing-type equation are given by Theorem~\ref{familyOfSolutionsWithoutLagrange}. The equivalent statement for the solutions to the alternative Killing-type equation is simply obtained by replacing equations~(\ref{strongSolutionCondition}--\ref{strongSolutionGauge}) with
\begin{gather}\label{strongSolutionConditionAlternative}
  \xi= -g^{-1}\partial_{\dot q}N,\\
  f=\tau L+N-\partial_{\dot q}L\cdot g^{-1}\partial_{\dot q}N.
  \label{strongSolutionGaugeAlternative}
\end{gather}
If we choose $\tau$ so as to get trivial $f\equiv0$, we obtain the alternative strong solution
\begin{equation}
  \label{strongSolutionConditionAlternativeWithTrivialGauge}
  \tau=-\frac{1}{L}\Bigl(N-
  \partial_{\dot q}L\cdot g^{-1}\partial_{\dot q}N\Bigr),\qquad
  \xi= -g^{-1}\partial_{\dot q}N\qquad
  f=0.
\end{equation}
Compare with equation~(I.17) from Candotti, Palmieri and Vitale~\cite{Candotti}.


\end{document}